\documentclass[twocolumn]{autart}    

\usepackage{latexsym,amssymb}
\usepackage{mathrsfs}
\usepackage{epsfig}
\usepackage{graphicx}
\usepackage{color}
\usepackage{enumerate}
\usepackage{amsmath}
\usepackage{breqn}
\usepackage{times}

\newtheorem{theorem}{Theorem}

\newtheorem{proposition}{Proposition}
\newtheorem{assumption}{Assumption}
\newtheorem{lemma}{Lemma}
\newtheorem{remark}{Remark}

\newenvironment{proof}{\medskip\noindent{\it Proof. }}{ \medskip}
\def\begequarr{\begin{eqnarray}}
\def\endequarr{\end{eqnarray}}
\def\begequarrs{\begin{eqnarray*}}
\def\endequarrs{\end{eqnarray*}}
\def\begequ{\begin{equation}}
\def\endequ{\end{equation}}
\def\begequs{\begin{equation*}}
\def\endequs{\end{equation*}}
\def\begite{\begin{itemize}}
\def\endite{\end{itemize}}

\def\begcen{\begin{center}}
\def\endcen{\end{center}}
\def\begrem{\begin{remark}\rm}
\def\endrem{\end{remark}}
\def\ba{\begin{array}}
\def\ea{\end{array}}

\def\col{ \mbox{col}\; }

\def\dst{\displaystyle}

\def\beeq#1{\begin{equation}{#1}\end{equation}}

\begin{document}

\begin{frontmatter}

\title{Robust Implementable Regulator Design of General Linear Systems} 

\vspace{-2.5em}
\author[UNSD]{Lei Wang}\ead{lei.wang2@sydney.edu.au},
\author[Bol]{Lorenzo Marconi}\ead{lorenzo.marconi@unibo.it},
\author[UoN]{Christopher M. Kellett}\ead{chris.kellett@anu.edu.au}

\vspace{-0.5em}
\address[UNSD]{Australia Centre for Field Robotics, The University of Sydney, Australia.}
\vspace{-0.5em}
\address[Bol]{C.A.S.Y - DEI, University of Bologna, Italy.}
\vspace{-0.5em}
\address[UoN]{Research School of Electrical, Energy, and Materials Engineering, Australian National University, Australia.}


\vspace{-1.5em}
\begin{abstract}                          
Robust implementable output regulator design approaches are studied for general linear continuous-time \mbox{systems} with periodically sampled measurements, consisting of both the regulation errors and extra measurements that are generally non-vanishing in steady state. A digital regulator is first developed via the conventional emulation-based approach, rendering the regulation errors asymptotically bounded with a small sampling period. We then develop a hybrid design framework by incorporating a generalized  hold device, which transforms the original problem into the problem of designing an output feedback controller fulfilling two conditions for a discrete-time system. We show that such a controller can always be obtained by designing a discrete-time internal model, a discrete-time washout filter, and a discrete-time output feedback stabilizer. As a result, the regulation errors are shown to be globally exponentially convergent to zero, while the sampling period is fixed but can be arbitrarily large.
This design framework is further developed for a multi-rate digital regulator with a large sampling period of the measurements and a small control execution period.
\end{abstract}

\end{frontmatter}

\section{Introduction}

The problem of output regulation is to design a controller so as to achieve asymptotic trajectory tracking and/or disturbance compensation in the presence of reference/disturbance signals that are trajectories of an autonomous system (the so-called exosystem).
Taking robustness into consideration, the internal model principle has been shown as the most effective design and analysis tool in the seminal work \cite{Francis1976} for linear continuous-time systems. Internal model-based methods have been well developed for continuous-time nonlinear systems with continuous measurements (e.g., \cite{Isidori1990,Huang2004,Wang2020}), hybrid systems (e.g., \cite{Marconi2013,Forte2017,Carnevale2017}) and networked systems (e.g., \cite{WangAut2019,Isidori2014}).

In general, the internal model-based regulator consists of two main components: the internal model  and the stabilizer. There are two kinds of control architectures, referred to as post-processing and pre-processing schemes (see e.g. \cite{Isidori2012}), depending on the topology used to connect the internal model and stabilizing units. As for the pre-processing scheme, the stabilizer is directly cascaded with the controlled plant, processing the regulation errors. In post-processing schemes, on the other hand,  the internal model is cascaded with the plant by processing the regulation errors. For single-input single-error (SISE) systems, both schemes are fundamentally equivalent with  pre-processing schemes that have been shown to be more constructive in some cases, such as for nonlinear systems with non-vanishing extra measurements (e.g. \cite{Wang2020,Toledo2006}), and multi-rate systems (e.g. \cite{Antunes2014}).

The measurements for feedback design are generally obtained from periodically sampling sensors, whose measure is not accessible continuously, and the regulator is practically implemented by digital devices or in combination with some simple analog devices, for example generalized zero-order hold devices.  In \cite{Castillo1997} for linear continuous-time systems, a state-feedback solution is studied locally by proposing a fully discrete-time regulator, which fulfills the internal model principle solely at the sampling time by thus guaranteeing only practical regulation. To achieve asymptotic regulation, in \cite{Lawrence2001} a hybrid internal model is proposed. It is shown that the continuous-time steady state input can be generated by means of continuous-time internal models  acting as ``generalized signal re-constructors" and there always exists a discrete-time stabilizer achieving the desired regulation purpose, though in the absence of robustness analysis.
Motivated by this work, \cite{Marconi2013} further develops a robust solution for SISE linear systems. In addition, \cite{Astolfi2018ECC,Astolfi2019Nolcos,Liu2020} adopt the emulation-based method by sampling the measurements and the control inputs in the absence of samples, which requires the sampling period to be small and renders the regulation errors bounded in a general sense.
In all the aforementioned results, the controlled continuous-time systems are required to be detectable by the regulation errors, which might not be fulfilled in practice, for example, as with the inverted pendulum on a cart considered in Section \ref{sec-5} below.

Motivated by the previous analysis, this paper studies the robust sampled-data regulation problem of \emph{general} linear continuous-time systems, for which the detectability property is fulfilled by the whole set of measurements, consisting of both the regulation errors and possible extra non-vanishing measurements. Motivated by \cite{Astolfi2018ECC,Astolfi2019Nolcos}, we first present an emulation-based solution, which under a small sampling period, renders the regulation errors asymptotically bounded by a constant, depending on the time derivative of the steady states of both extra measurements and control inputs, and adjustable by the sampling period.  In view of this, we then develop a hybrid design framework by incorporating a generalized  hold device. By cascading this device with the controlled plant, it is shown that the desired regulation objective can be achieved by designing a discrete-time output feedback stabilizer fulfilling two conditions, i.e., stabilizing the closed loop at the origin and compensating for the steady state input. Inspired by \cite{Wang2020}, to fulfill both conditions, we further propose a discrete-time internal model and a discrete-time washout filter, which in turn simplifies the problem to the design of a discrete-time output feedback stabilizer for an augmented discrete-time linear system that is stabilizable and detectable.
As a result, the regulation errors are shown to be globally exponentially convergent to zero, while the sampling period is fixed but can be arbitrarily large. By regarding both the generalized  hold device and the discrete-time internal model as a hybrid internal model unit, we note that the proposed robust implementable regulator is consistent with the pre-processing internal model-based
structure proposed in \cite{Wang2020}.
Furthermore, this design framework is developed for a multi-rate digital regulator with a large sampling period of the measurements and a small control execution period. We show that given almost any large sampling period of the measurements, the regulation errors can be rendered to be bounded by a constant, depending on the time derivative of the desired steady states of control inputs, and adjustable by the control execution period.

This paper is organized as follows. In Section 2 the considered problem is explicitly formulated and some standing assumptions are presented. Section 3 presents the emulation-based approach, which then motivates us to propose a new implementable regulator design framework via generalized hold devices in Section \ref{sec-4}. In Section \ref{sec-5}, this framework is further explored for a multi-rate digital regulator. To show the effectiveness of the proposed approach, the linearly approximated model of the inverted pendulum on the cart is studied in Section 6. Conclusions are presented in Section 7.
This paper is different from the conference version \cite{Wang2020IFAC} by additionally presenting the motivating emulation-based approach in Section 3 and developing the multi-rate digital regulator in Section \ref{sec-5}.

\section{Problem Statement}
\label{Ssec-problem}
Consider the output feedback regulation problem for linear systems
\beeq{\label{ini-sys}\ba{rcl}
\dot w &=& Sw\,\\
\dot x &=& A\, x + B\, u + P\,w\,\\
y &=& C\,x + Q\,w
\ea}
with exogenous states $w\in\mathbb{R}^d$, states $x\in\mathbb{R}^{n}$, inputs $u\in\mathbb{R}^m$ and measurements $y\in\mathbb{R}^{q}$.
We deal with a \emph{general} class of linear systems in which the measurements $y$ consist of regulation errors $e(t):=C_ex(t)+ Q_ew(t)\in\mathbb{R}^{q_e}$, to be steered to zero asymptotically, and also extra measurements $y_m(t):=C_mx(t) + Q_mw(t)\in\mathbb{R}^{q_m}$ on which no specific regulation requirements are fixed, with $q=q_e+q_m$. Considering the practical situation in which the measurements are generally obtained in a discrete-time manner, i.e.,  periodically sampled with  sample time $T>0$, the measurements available for feedback are given by the sampled regulation errors
$\hat e(t):=C_ex(t_k)+ Q_ew(t_k)$ and the sampled extra measurements $\hat y_m(t):=C_mx(t_k) + Q_mw(t_k)$ for $t\in[t_k,t_{k+1})$, $t_k=kT$ and $k\in\mathbb{N}$. As customary in the field of output regulation, we assume that $S$ is neutrally stable and there exists an invariant compact set $\mathcal{W}\in\mathbb{R}^d$ such that $w(t)\in\mathcal{W}$ for all $t\geq0$. For convenience, we set $|\mathcal{W}|:=\max_{w\in\mathcal{W}}\|w\|$.

In this setting, the control objective is to design a \emph{robust implementable} regulator driven by the sampled measurements $(\hat e(t),\hat y_m(t))$ such that the resulting closed-loop trajectories are bounded, and the \emph{continuous-time} regulation errors $e(t)$ asymptotically converge to zero. As in \cite{Francis1976,Wang2020}, we are interested in a \emph{robust} solution, i.e., the above control objective is  guaranteed even if all system  matrices of (\ref{ini-sys})  except $S$ vary in a (small) neighborhood of their nominal forms\footnote{If there exist uncertainties on the matrix $S$, then the idea of an adaptive internal model (e.g. \cite{Serrani2001,WangTAC2019}) can be employed.}.
Additionally, this paper presents an \emph{implementable} solution that can be directly implemented by purely digital devices, or together with some simple analog devices, such as generalized  hold devices (see \cite{Lawrence2001}).

Due to the presence of the sampled measurements, the resulting system is fundamentally hybrid. In this paper, we will follow terminologies from \cite{Goebel2012} to denote the time $t\in[t_k,t_{k+1})$ by a hybrid time domain $(t,k)$, and represent a hybrid system as a combination of  flow  and  jump dynamics, which are respectively described by  differential  and  difference equations. The action of sampling the measurements $e$ and $y_m$ leads to a measurement model of the kind
\beeq{\label{eq:measurements}
\ba{l}
\left\{
\ba{rcl}
\dot \tau &=& 1\,\\
\dot {\hat e} &=& 0\,\\
\dot {\hat y}_m &=& 0\,
\ea
\right. \,,\qquad \qquad \qquad \mbox{for }(\tau,\hat e,\hat y_m)\in[0,T)\times\mathbb{R}^{q}\\
\left\{
\ba{rcl}
\tau^+ &=& 0\,\\
\hat e^+ &=& C_ex+Q_ew \,\\
\hat y_m^+ &=& C_mx+Q_mw\,
\ea
\right. \,,\quad \mbox{for }(\tau,\hat e,\hat y_m)\in \{T\}\times\mathbb{R}^{q}
\ea}
in which $\tau$ is a clock state, and $\hat e,\hat y_m$ are the sampled measurements available for feedback.
In the subsequent Sections 3 and 4, the flow and jump conditions are governed by the clock $\tau$ only as in (\ref{eq:measurements}), i.e., the flow occurs for $\tau\in[0,T)$ and the jump occurs for $\tau=T$, which will be occasionally omitted for simplicity.

Throughout this paper, for any square matrix $M$, we denote by $\sigma(M)$ its spectrum. We make some standard assumptions previously used for a robust continuous-time solution (e.g. \cite{Francis1976}).
\begin{assumption}\label{ass-1}
$\,$
\begin{itemize}\vspace{-1em}
  \item[(i)] The matrix triplet $(A,B,C)$ is stabilizable and detectable;
  \item[(ii)] There holds the non-resonance condition
\beeq{\label{NRC}
\mbox{rank}\begin{bmatrix}
A-\lambda I_n & B \cr C_e & 0
\end{bmatrix}
= n+q_e\,,\quad \forall \lambda\in\sigma(S)\,.
}
\end{itemize}
\end{assumption}

Assumption \ref{ass-1} immediately implies that for any pair of matrices $(P,Q_e)$, there exist $\Pi_x\in\mathbb{R}^{n\times d}$ and $\Psi\in\mathbb{R}^{m\times d}$ such that  the regulator equations
\beeq{\label{RE-3}
\ba{rcl}
\Pi_x S &=& A\Pi_x + B\Psi + P\,\\
0 &=& C_e \Pi_x + Q_e\,
\ea
}
are satisfied.

%
As in \cite{Castillo1997,Lawrence2001}, in order to preserve the stabilizability and detectability of system (\ref{ini-sys}) after discretization,  the following assumption is made.
\vspace{-0.5em}
\begin{assumption}\label{ass-2}
The sampling period $T$ is not pathological from the pair $(A,S)$. That is, for any distinct $\lambda_i,\lambda_j\in\sigma(A)\bigcup\sigma(S)$,
\beeq{\label{eq:sample-T}
\lambda_i-\lambda_j \neq 2k\pi / T\,, \quad \mbox{ for any $k\in\mathbb{N}$}\,.
}
\end{assumption}

\vspace{-1.5em}
Note that condition (\ref{eq:sample-T}) is  generically satisfied for all sampling time $T\in\mathbb{R}_+$, given any matrices $A,S$.

\begin{remark}
We note that this paper follows the internal model-based regulator design framework, where solutions of the regulator equations (\ref{RE-3}) are not used for feedback design, and the internal model is designed to  compensate for the steady state input $\Psi w(t)$. Thus, as in other relevant works \cite{Francis1976,Huang2004,Wang2020}, the presented designs yield a robust regulation with respect to small variations of the system matrices of (\ref{ini-sys})  except $S$.
\end{remark}

\section{An Emulation-based Approach}
\label{sec-MR}

Before presenting the proposed framework, this section aims to present how the conventional emulation-based approach \cite{Astolfi2018ECC,Astolfi2019Nolcos} can be applied to solve the considered problem.

To apply the emulation-based approach, we first assume that the \emph{continuous-time} measurements $y(t)$ are accessible and investigate a robust regulator driven by $y(t)=(e(t),y_m(t))$. Typically, there are two internal model-based control schemes: post-processing \cite{Francis1976} and pre-processing \cite{Wang2020} methods, both leading to a robust dynamical regulator of the form
\vspace{-1em}\beeq{\label{eq:regulator-Em}
\ba{rcl}
\dot x_c &=& A_c x_c + B_c y\,,\quad x_c\in\mathbb{R}^{n_c}\\
u &=& C_c x_c\,
\vspace{-0.5em}\ea}
where matrices $A_c\in\mathbb{R}^{n_c\times n_c},B_c\in\mathbb{R}^{n_c\times q},C_c\in\mathbb{R}^{m\times n_c}$ are designed such that the following two requirements hold (always satisfiable under Assumption 1):
\begin{itemize}\vspace{-0.8em}
  \item[(i)] the matrix
  $
  \mathbf{A} := \begin{bmatrix} A & BC_c \cr B_c C & A_c\end{bmatrix}
  $
  is Hurwitz;
  \item[(ii)] for all $(\Pi_x,\Psi)$ satisfying (\ref{RE-3}), the matrix equations
  \vspace{-0.5em}\[\vspace{-0.5em}\ba{l}
  \Pi_{x_c} S = A_c \Pi_{x_c} + B_c \begin{bmatrix}0\cr Y_m\end{bmatrix}\,,\qquad
  \Psi = C_c\Pi_{x_c}
  \ea
  \vspace{-0.5em}\]
  with $Y_m=C_m\Pi_x + Q_m$, admit a solution $\Pi_{x_c}$.
\end{itemize}

With these matrices $A_c,B_c,C_c$ in hand, following the emulation-based method, it is always possible to design a regulator driven by the \emph{sampled} measurements $\hat y = (\hat e,\hat y_m)$ fulfilling the dynamics (\ref{eq:measurements}), of the form
\vspace{-0.8em}\beeq{\label{eq:regulator-Em1}
\ba{l}
\dot \tau = 1,\quad
\dot x_c = A_c x_c + B_c \hat y ,\quad
\dot { u} = 0\,\\
\qquad \qquad\mbox{for }(\tau, x_c,u)\in[0,T)\times\mathbb{R}^{n_c+m}\,,\\
\tau^+ = 0,\quad
 x_c^+ =  x_c,\quad
{u}^+ = C_c x_c\\
\qquad \qquad\mbox{for }(\tau, x_c,u)\in\{T\}\times\mathbb{R}^{n_c+m}.
\ea}
By setting ${\hat x}_c := e^{-A_c\tau}x_c - \dst\int_0^\tau e^{-A_c\,r}{\rm d} r \,B_c \,\hat y$,  (\ref{eq:regulator-Em1}) can be rewritten into a discrete-time equivalent form
\beeq{\label{eq:regulator-Em2}
\ba{l}
\dot {\hat x}_c = 0 \,,\qquad \dot {u} = 0\\
 \hat x_c^+ =  e^{A_cT}\hat x_c - \int_0^T e^{-A_c\,(r-T)}{\rm d} r B_c \hat y \\
{u}^+ = C_c [e^{A_cT}\hat x_c - \int_0^T e^{-A_c\,(r-T)}{\rm d} r B_c \hat y]\,,
\ea}
which clearly can be implemented by digital devices.

Fundamental to show the properties of the closed-loop system (\ref{ini-sys})-(\ref{eq:regulator-Em2}) is the following lemma, which is adapted from \cite[Lemma 2]{Astolfi2018ECC}.
\begin{lemma}\label{lemma-3.1}
  There exist a symmetric positive definite matrix $\mathbf{P}\in\mathbb{R}^{(n+n_c)\times(n+n_c)}$, $\kappa>0$ and $\gamma^\ast>0$ such that
  \vspace{-1em}\beeq{
  \begin{bmatrix}
  \mathbf{P}\mathbf{A}+\mathbf{A}\mathbf{P} + 2\kappa \mathbf{P} + \frac{2}{\gamma^\ast}\mathbf{C}^\top\mathbf{C} & \mathbf{P}\mathbf{B} \cr
  \mathbf{B}^\top \mathbf{P} & -\gamma^\ast I_q
  \end{bmatrix}\leq 0
  }
  with $\mathbf{B}=\begin{bmatrix}0 & B\cr B_c & 0\end{bmatrix}$ and $\mathbf{C}=\begin{bmatrix}CA & CBC_c \cr C_cB_cC & C_cA_c\end{bmatrix}$.
\end{lemma}

Let
\beeq{\label{eq:tau_max}
\tau_{\max}(\kappa,\gamma) := \left\{\ba{ll}
\frac{1}{\kappa r}\arctan(r)\,,\qquad  & \mbox{if }\gamma >\kappa\,\\
\frac{1}{\kappa}\,,\qquad & \mbox{if }\gamma =\kappa\,\\
\frac{1}{\kappa r}\mbox{arctanh}(r)\,,\qquad & \mbox{if } \gamma <\kappa\,
\ea\right.
}
where
$
r = \sqrt{|\frac{\gamma^2}{\kappa^2}-1|}\,.
$
\begin{proposition}\label{proposition-1}
   Suppose Assumption \ref{ass-1} holds. Choose $\kappa,\gamma^\ast$ according to Lemma \ref{lemma-3.1}. Then for all $\gamma\geq\gamma^\ast$, and $T\in(0,\tau_{\max}(\kappa,\gamma)]$,
   \vspace{-0.8em}\begin{itemize}
     \item[(i)]the trajectories of the resulting closed-loop system (\ref{ini-sys}) with regulator (\ref{eq:regulator-Em2}) are bounded, and
     \item[(ii)] there exists a $k_{\rm em}>0$, independent of $\gamma,\kappa$ such that
     \vspace{-0.8em}
     \beeq{\label{eq:emu-e}
  \lim_{t+j\rightarrow\infty}\|e(t,j)\| \leq \frac{k_{\rm em}}{\sqrt{\gamma\kappa}}\left\|\begin{bmatrix}Y_mS\cr \Psi S\end{bmatrix}\right\|.
  }
   \end{itemize}
\end{proposition}

\vspace{-0.5em}
We observe that the emulation-based regulator (\ref{eq:regulator-Em2}) renders the closed-loop trajectories bounded, provided that the sampling period $T$ is smaller than $\tau_{\max}(\kappa,\gamma)$. That is, this method may not be effective when a large $T$ is desired. On the other hand, from (\ref{eq:emu-e}) it is seen that the regulation error $e(t)$ is asymptotically bounded by a constant, depending on the time-derivative of the desired steady states of $y_m$ and $u$, and on the sampling interval (i.e., $\gamma,\kappa$). This bound can be arbitrarily decreased by increasing $\gamma$, which in turn implies a smaller upper bound for allowable $T$. Practical (and not asymptotic) regulation is essentially due to the fact that the control input is also sampled and thus the continuous-time steady state control input, which is $\Psi w(t)$, is not exactly generated by the controller (\ref{eq:regulator-Em1}). Motivated by these restrictions, we develop a design technique, which can not only solve the problem under a very large $T$, but also render the regulation errors $e(t)$ exponentially convergent to zero.

\section{Robust Regulator Design Using Generalized Hold Devices}
\label{sec-4}
\subsection{Problem Transformation}

\vspace{-1em}
It is well-known (see \cite{Francis1976}) that the steady state input forcing the desired regulation objective of system (\ref{ini-sys}) is a continuous-time signal of the form $u_{ss}(t)=\Psi w(t)$ with $\Psi$ provided by the regulator equations (\ref{RE-3}).
Note that, in general, this continuous-time steady state input cannot be perfectly re-constructed by a discrete-time compensator. This  motivates us to embed into the regulator a continuous-time signal reconstructor, which generates the steady-state  control input during flows. Motivated by \cite{Lawrence2001}, we deal with a cascade of the controlled plant (\ref{ini-sys}) and a \emph{generalized hold device}, the latter described by
\vspace{-0.em}\beeq{\label{zhd}
  \dot \zeta = (\Phi\otimes I_{q_e})\zeta \,,\qquad
  \zeta^+ =  v_\zeta\,
}
with state $\zeta\in\mathbb{R}^{dq_e}$ and input $v_\zeta\in\mathbb{R}^{dq_e}$ used to reset $\zeta$ at the sampling time, and  $\Phi\in \mathbb{R}^{d\times d}$ being a matrix whose  minimal polynomial is coincident with that of $S$, the latter denoted by
$
\mathcal{P}_S(\lambda)=\sum_{i=0}^{d-1}s_i\lambda^i + \lambda^{d}\,.
$
To ease subsequent analysis, there is no loss of generality to let
\vspace{-0.5em}\[
\Phi = \begin{bmatrix}{\bf 0}_{d-1} & I_{d-1} \cr -s_0 & \left(-s_1\, \cdots\, -s_{d-1} \right)\end{bmatrix}\,,
\]\vspace{-1em}
with ${\bf 0}_{d-1}$ being a zero column vector of dimension $d-1$.

The feedback control law is designed as
\vspace{-0.0em}\beeq{\label{u}
   u = L \zeta +  v_u
}
where $v_u\in\mathbb{R}^m$ denotes the residual control input that will be designed as digital, i.e., $\dot v_u=0$, and $L\in\mathbb{R}^{m\times dq_e}$ is such that the pair $(\Phi\otimes I_{q_e}\,,L)$ is observable.

By augmenting (\ref{zhd}) and (\ref{u}) with (\ref{ini-sys}), we  obtain
\beeq{\label{ex-sys-f}
\left\{\ba{rcl}
\dot w &=& S\,w\,\\
\dot { x} &=& A\, x + B\, L\, \zeta + B\,v_u + P\,w \,\\
\dot {\zeta} &=& (\Phi\otimes I_{q_e})\,\zeta
\ea\right.
}
during the flow, and during the jump,
\beeq{\label{ex-sys-j}
 w^+ = w \,,\quad
 x^+ = x \,,\quad
\zeta^+ =  v_\zeta\,
}
with input $v:=\col(v_u,v_{\zeta})$ satisfying $\dot v=0$ during flow.

The problem at hand is thus to design a digital controller of the hybrid system (\ref{ex-sys-f})-(\ref{ex-sys-j}) with control inputs $v$ and measurements $(\hat e,\hat y_m)$ fulfilling the dynamics (\ref{eq:measurements}) such that the resulting closed-loop system trajectories are bounded and $\lim_{t+j\rightarrow\infty}e(t,j)=0$.
By letting
\beeq{\label{eq:def-xD}\ba{l}
w_D(\tau) := e^{-S\tau}w\,\\
x_D(\tau) := e^{-A\tau}x - \int_0^\tau e^{-A\,r}BL e^{(\Phi\otimes I_{q_e})(r-\tau)}{\rm d} r \,\zeta \\ \qquad - \int_0^\tau e^{-A\,r}{\rm d} r \,B \,v_u - \int_0^\tau e^{-A\,r}Pe^{S\,(r-\tau)}{\rm d} r \,w \,\\
\zeta_{D}(\tau) :=e^{-(\Phi\otimes I_{q_e})\tau}\,\zeta\,,
\ea }
it immediately follows that the ``sample-data" discrete-time system linked to (\ref{ex-sys-f})-(\ref{ex-sys-j}) is given by
\beeq{\ba{l}\label{E-DTS}
\dot w_D=0\,,\quad \dot x_{D} =0 \,,\quad \dot \zeta_{D} =0\,\\
\left\{
\ba{rcl}
w_D^+ &=& S_D \, w_D\,\\
x_{D}^+ &=& A_{D}\, x_{D} + L_{D}\,\zeta_{D} + B_{D}\, v_{u} + P_D w_D\,\\
\zeta_{D}^+ &=&  v_{\zeta}\,\\
\ea\right.\\
\hat y=\begin{bmatrix}\hat e \cr \hat y_{m}\end{bmatrix} = \begin{bmatrix} C_e x_D + Q_e w_D \cr C_m x_D + Q_m w_D \end{bmatrix}
\ea}
with $S_D := e^{S\,T}$, $A_{D} := e^{A\,T}$, $B_{D}:=\int_0^Te^{A\,r}{\rm d} r\, B$, $L_{D}: =\int_0^Te^{A\,(T-r)}BL e^{(\Phi\otimes I_{q_e})r}{\rm d} r$,  and $P_{D}:=\int_0^Te^{A\,(T-r)}\cdot$ $\cdot Pe^{S\,r}{\rm d} r$.
For system (\ref{E-DTS}) the following holds.
\begin{lemma}
Suppose Assumptions \ref{ass-1} and \ref{ass-2} hold.
\vspace{-1em}\begin{itemize}
  \item[(i)] The system (\ref{E-DTS}) is stabilizable and detectable w.r.t. inputs $v$ and outputs $\hat y$ when $w_D=0$.
  \item[(ii)] Let $\Phi_D:=e^{\Phi\,T}$. For  any $\Pi_x$ in (\ref{RE-3}), there exists $\Pi_\zeta\in\mathbb{R}^{dq_e\times d}$ such that
\beeq{\label{REs-2}
\ba{rcl}
\Pi_x S_D &=& A_D\Pi_x + L_{D}\,\Pi_\zeta  + P_D\,\\
\Pi_\zeta S_D &=& (\Phi_D\otimes I_{q_e})\Pi_\zeta\,\\
0 &=& C_e \Pi_x + Q_e\,.
\ea
}
\end{itemize}
\vspace{-1em}
\end{lemma}
\begin{proof}
\emph{Proof of (i).}With Assumption \ref{ass-1}.(i) and \ref{ass-2}, according to \cite{Kimura1990}, it can be deduced that $(A_D,B_D,C)$ is stabilizable and detectable. Then using the PBH test, it can be  verified that system (\ref{E-DTS}) is stabilizable and detectable w.r.t. inputs $v$ and outputs $\hat y$ when $w_D=0$.

\emph{Proof of (ii).}
As for the solution of (\ref{REs-2}), we observe that, for any $\Psi\in\mathbb{R}^{m\times d}$, since $(\Phi\otimes I_{q_e},L)$ is observable, there always exists a unique solution $\Pi_\zeta$ such that
\beeq{\label{RE-4}
\ba{rcl}
\Pi_\zeta S &=& (\Phi\otimes I_{q_e}) \Pi_\zeta\,\\
\Psi &=& L \Pi_\zeta\,.
\ea}
In view of the fact that (\ref{REs-2}) is  the discretized form of  (\ref{RE-3}) (derived by Assumption \ref{ass-1}.(ii)) and (\ref{RE-4}), this indicates that such $(\Pi_x,\Pi_\zeta)$ is also the solution of (\ref{REs-2}), completing the proof. $\blacksquare$
\end{proof}

The desired robust regulator can be completed by designing an output feedback controller for the discrete-time system (\ref{E-DTS}), having the form
\beeq{\label{DR}\ba{l}
\dot z = 0\,,\quad
z^+ = A_z\, z + B_z\, \hat y \,,\quad z_k\in\mathbb{R}^{n_z}\,,\\
v = K_z \,z + L_z\, \hat y\,.
\ea
}

\begin{theorem}\label{theo-1}
Suppose Assumptions \ref{ass-1} and \ref{ass-2} hold. The robust sampled output regulation problem of system (\ref{ini-sys}) is solved by the regulator (\ref{zhd}), (\ref{u}), and (\ref{DR}) \emph{if}
\begin{itemize}
  \item[(a)] the origin of the closed-loop discrete-time system (\ref{E-DTS}), (\ref{DR}) with $w_D=0$ is globally exponentially stable;
  \item[(b)] for any $\Pi_x,\Pi_\zeta$ satisfying (\ref{REs-2}), letting $Y_m=C_m\Pi_x+Q_m$, there exists a solution $\Pi_z\in\mathbb{R}^{n_z\times d}$ of
\beeq{\label{RE-z}\ba{rcl}
\Pi_z S_D &=& A_z\, \Pi_z + B_z\, \begin{bmatrix}{\bf 0}_{d\times q_e} & Y_m^{\top}\end{bmatrix}^{\top}\,\\
\begin{bmatrix}{\bf 0}_{m\times d}\cr \Pi_\zeta S_D\end{bmatrix} &=& K_z \,\Pi_z + L_z\, \begin{bmatrix}{\bf 0}_{d\times q_d} & Y_m^{\top}\end{bmatrix}^{\top}\,.
\ea}
\end{itemize}
\end{theorem}

In the following lemma, we claim that if the number of  control inputs matches the one of regulation errors  then the above sufficient conditions are also necessary.

\begin{lemma}\label{coro-1}
Suppose Assumptions \ref{ass-1} and \ref{ass-2} hold. The sampled robust output regulation problem of system (\ref{ini-sys}) with $m=q_e$ is solved by the regulator (\ref{zhd}), (\ref{u}), and (\ref{DR}) \emph{if and only if} the requirements (a) and (b) in Theorem \ref{theo-1} are fulfilled.
\end{lemma}

The proofs of Theorem \ref{theo-1} and Lemma \ref{coro-1} are respectively given in Appendices \ref{proof-theo-1} and \ref{proof-coro-1}.

\subsection{About the Design of the Controller (\ref{DR})}

In the previous subsection, we have shown that the original robust output regulation problem with sampled measurements can be transformed into the design of a discrete-time output feedback controller (\ref{DR}) such that the requirements (a) and (b) in Theorem \ref{theo-1} are fulfilled.

Observe that by partitioning $K_z$ and $L_z$ as $\begin{bmatrix}K_{z,u} \cr K_{z,\zeta}\end{bmatrix}$ and $\begin{bmatrix}L_{z,u} \cr L_{z,\zeta}\end{bmatrix}$ consistently with the partition of $v$ in $v_u$ and $v_{\zeta}$, the second equation of (\ref{RE-z}) can be partitioned into two parts
\begin{eqnarray}
{\bf 0}_{m\times d} &=& K_{z,u} \Pi_z + L_{z,u} \begin{bmatrix}{\bf 0}_{d\times q_e} & Y_m^{\top} \end{bmatrix}^{\top}\, \label{eq:0}\\
\Pi_\zeta S_D &=& K_{z,\zeta} \,\Pi_z + L_{z,\zeta}\,  \begin{bmatrix}{\bf 0}_{d\times q_e} & Y_m^{\top}\end{bmatrix}^{\top}\label{eq:Pi_zeta}\,.
\end{eqnarray}
From (\ref{eq:0}), one can see that the controller (\ref{DR}) is required to block the steady state of the extra measurements $\hat y_{m}$, denoted by $Y_mw_D$. In addition, (\ref{eq:Pi_zeta}) expresses the ability  of  the controller (\ref{DR})  to reproduce the ideal steady state of input $v_{\zeta}$, which is $\Pi_{\zeta}S_Dw_D$.

In the following, we show how to systematically design the controller (\ref{DR}) so that the conditions of Theorem \ref{theo-1} are fulfilled.
Motivated by \cite{Wang2020}, the fulfillment of (\ref{eq:0}) suggests the design of a washout digital filter to $\hat y_{m}$ so as to block its steady state. The digital filter takes the form
\beeq{\label{filter}\ba{l}
\dot \xi = 0 \,,\quad
\xi^+ = F_{\rm f}\, \xi + G_{\rm f}\, \hat y_{m}\,\\
y_{\rm f} = \hat y_{m} - \Gamma_{\rm f}\, \xi
\ea}
where $y_{\rm f}$ is the filter output,  $(F_{\rm f},\Gamma_{\rm f}) \in \mathbb{R}^{dq_m\times dq_m} \times \mathbb{R}^{q_m\times dq_m}$ is an observable pair with all eigenvalues of $F_{\rm f}$ lying strictly within the unit circle, and $G_{\rm f}\in\mathbb{R}^{dq_m\times q_m}$ is such that $\Phi_D\otimes I_{q_m}=F_{\rm f}+G_{\rm f} \Gamma_{\rm f}$.

Similarly, the fulfillment of (\ref{eq:Pi_zeta}) suggests to consider a discrete-time internal model of the form
\beeq{\ba{l}\label{IMPre}
\dot \eta = 0\,,\quad
\eta^+ = (\Phi_D\otimes I_{q_e})\eta + v_\eta\,\\
v_{\zeta} = \eta + \bar v_{\zeta}
\ea}
with $\eta\in\mathbb{R}^{dq_e}$ and $\bar v_{\zeta}, v_\eta\in \mathbb{R}^{dq_e}$ to be determined later.

\begin{lemma}\label{lemma-2}
Suppose Assumptions \ref{ass-1} and \ref{ass-2} hold. The augmented discrete-time system (\ref{E-DTS}), (\ref{filter}), and (\ref{IMPre}) is stabilizable and detectable  w.r.t. inputs $\bar v := (v_{u},\bar v_{\zeta},v_{\eta})$ and outputs $\mathrm{y}:=\col(\hat e, y_{\rm f})$, when $w_D=0$.
\end{lemma}
The proof of Lemma \ref{lemma-2} is given in Appendix \ref{app-lemma-2}. This lemma naturally suggests to  design a discrete-time output feedback stabilizer for the augmented discrete-time system (\ref{E-DTS}), (\ref{filter}), and (\ref{IMPre}) when $w_D=0$. Let
\beeq{\label{vartheta}\ba{l}
\dot \vartheta = 0\,,\quad
\vartheta^+ = A_\vartheta\,\vartheta + B_\vartheta \mathrm{y}\,,\quad \vartheta\in\mathbb{R}^{n_z-dq}\\
\bar v  = C_\vartheta\,\vartheta + D_\vartheta\,\mathrm{y}\,.
\ea}
be such a controller. The following theorem can be then proved.
\begin{theorem}\label{theo-2}
Let (\ref{vartheta}) be a stabilizer for the augmented system (\ref{E-DTS}), (\ref{filter}), and (\ref{IMPre}). Then, with system (\ref{DR}) defined by the cascade of (\ref{filter}), (\ref{IMPre}),  and (\ref{vartheta}), the requirements (a) and (b) in Theorem \ref{theo-1} are fulfilled.
\end{theorem}
\begin{proof}
  It is clear that by designing the controller (\ref{DR}) as the cascade of (\ref{filter}), (\ref{IMPre}) and (\ref{vartheta}), the requirement (a) in Theorem \ref{theo-1} is fulfilled. Now we proceed to show that the requirement (b) is also fulfilled.

We first note that since $(F_{\rm f},\Gamma_{\rm f})$ is observable and $\Phi_D\otimes I_{q_m}=F_{\rm f}+G_{\rm f} \Gamma_{\rm f}$, given any $Y_m\in\mathbb{R}^{q_m\times d}$, there exists a unique solution $\Pi_{\rm f}\in\mathbb{R}^{q_md\times d}$ for the linear matrix equation
\beeq{\label{Re-f}\ba{rcl}
\Pi_{\rm f}\, S_{D} &=& F_{\rm f}\, \Pi_{\rm f} + G_{\rm f}\, Y_m\,\\
0 &=& Y_m  - \Gamma_{\rm f}\,\Pi_{\rm f}\,.
\ea}
Then, for any $\Pi_\zeta$ satisfying (\ref{REs-2}), we have
\beeq{\label{Re-im}\ba{rcl}
\Pi_{\eta}\, S_{D} &=& (\Phi_D\otimes I_{q_e})\, \Pi_{\eta}\,
\ea}
with $\Pi_\eta = \Pi_\zeta S_D$.
Thus, it can be easily verified that with the controller (\ref{DR}) as the cascade of (\ref{filter}), (\ref{IMPre}) and (\ref{vartheta}), the resulting matrix equations (\ref{RE-z}) permit a solution $\Pi_z=\begin{bmatrix}\Pi_{\rm f}^{\top} & \Pi_\eta^{\top} & {\bf 0}_{d\times (n_z-dq)}\end{bmatrix}^{\top}$, i.e., the resulting requirement (b) in Theorem \ref{theo-1} is fulfilled. $\blacksquare$
\end{proof}

We conclude the section by observing that the proposed controller fits in the ``pre-processing" structure delineated in \cite{Wang2020} and that, consistently with the prescriptions of \cite{Lawrence2001,Marconi2013}, the internal model unit contains $q_e$ copies of the continuous-time exosystem and $q_e$ copies of the discretized exosystem (see Figure \ref{fig1}).

\begin{figure}[thpb]
\begin{center}
\centering\includegraphics[height=30mm,width=70mm]{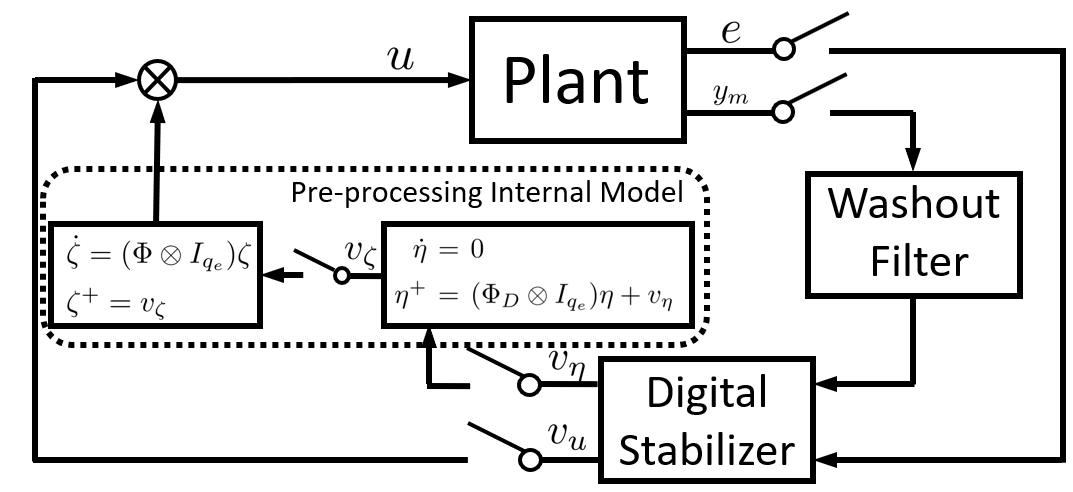} \caption{Structure of the proposed regulator.}
\label{fig1}
\end{center}
\end{figure}

\section{Multi-Rate Digital Regulator}
\label{sec-5}

In the previous section, we have proposed a robust hybrid regulator solving the considered problem globally and exponentially, even if the measurements are sampled with a very large interval. The regulator is implemented in a hybrid manner, i.e., a combination of a digital controller and a generalized hold device. In this section, this hybrid regulator is further developed for a pure digital regulator. Recalling that the sampling interval of measurements can be almost arbitrarily large in Theorem \ref{theo-2}, we thus adopt multi-rate samplings for system (\ref{ini-sys}), i.e., the control execution period is $\frac{T}{N}$ where $T$, that can be very large, is the sampling period of measurements and parameter $N\in\mathbb{N}_+$ is used to determine the control execution interval. In this respect, the control signal $u$ can be modelled by the following hybrid form
\beeq{\label{eq:control}
\ba{l}
\left\{
\ba{rcl}
\dot \tau_c &=& 1\,\\
\dot u &=& 0\,\\
\ea
\right. \,,\qquad  \qquad \mbox{for }(\tau_c,u)\in[0,\frac{T}{N})\times\mathbb{R}^m\\
\left\{
\ba{rcl}
\tau_c^+ &=& 0\\
u^+ &=& \hat u \\
\ea
\right. ,\qquad \qquad  \mbox{for }(\tau_c,u)\in \{\frac{T}{N}\}\times\mathbb{R}^m\,
\ea}
where $\tau_c$ is a clock state and $\hat u$ is an input to be determined.

Following the design in Section \ref{sec-4}, we design $\hat u$ as
\beeq{\label{eq:hat_u}\ba{l}
\hat u = L\zeta + K_{z,u} z + L_{z,u}\hat y\\
\left\{\ba{l}
\dot \tau = 1 \\
\dot \zeta = (\Phi\otimes I_{q_e}) \zeta \\
\dot z =0\,\\
\ea\right., \qquad  \mbox{for }(\tau,\zeta,z)\in[0,T)\times\mathbb{R}^{n_c}\\
\left\{\ba{l}
\tau^+ = 0\\
\zeta^+ = K_{z,\zeta} z + L_{z,\zeta}\hat y\\
z^+ = A_z\, z + B_z\, \hat y\\
\ea\right.,   \mbox{for }(\tau,\zeta,z)\in\{T\}\times\mathbb{R}^{n_c}
\,\\
\ea}
with $(\Phi\otimes I_{q_e}\,,L)$ observable. Let $K_z=\begin{bmatrix}K_{z,u} \cr K_{z,\zeta}\end{bmatrix}$ and $L_z=\begin{bmatrix}L_{z,u} \cr L_{z,\zeta}\end{bmatrix}$. The resulting closed-loop stability is formulated below, with proof given in Appendix \ref{app-theo-3}.

\begin{theorem}\label{theo-3}
   Suppose that Assumption \ref{ass-1}  holds. Let $T>0$ be any number satisfying Assumption \ref{ass-2}, and choose $A_z,B_z,K_z,L_z$ according to Theorem \ref{theo-2}. Then there exist a $N^\ast\in\mathbb{N}_+$ and a $\gamma^\dag>0$ such that for all $N\geq N^\ast$, the trajectories of the resulting closed-loop system (\ref{ini-sys}) with a digital regulator (\ref{eq:control})-(\ref{eq:hat_u}) are bounded, and
   \beeq{\label{eq:e-bound}
  \lim_{t+j\rightarrow\infty}\|e(t,j)\| \leq \frac{\gamma^\dag}{\sqrt{N}}\| \Psi S\| \,.
  }
\end{theorem}

\begin{remark}
  In contrast with (\ref{eq:emu-e}) in Proposition \ref{proposition-1}, (\ref{eq:e-bound}) demonstrates that the regulation error $e$ eventually converges to a set in relation to the time derivative of the desired steady states of $u$, independent of that of $y_m$. In this respect, when the desired steady state of $u$ is constant and that of $y_m$ is time-varying, the digital regulator (\ref{eq:control})-(\ref{eq:hat_u}) can guarantee that the regulation errors exponentially converge to zero, while there is no such guarantee for the emulation-based approach by Proposition \ref{proposition-1}.
\end{remark}

\section{An Example}
\label{sec-6}

Consider the output regulation problem of an inverted pendulum on a cart \cite{Khalil2002,Kwakernaak1972}, whose linearly approximated model is described by
\beeq{\label{InPC}
\ba{rcl}
m_0\ddot q &=& -mg\theta - \mu_f\dot q + u + P_1w\,\\
m_0\ell \ddot\theta &=& (m_0+m)g\theta + \mu_f\dot q - u + P_2 w
\ea}
where $q$ is the distance of the cart from the zero reference, $\theta$ is the angle of the pendulum w.r.t. the vertical axis, input $u$ is the horizontal force applied to the cart, and $P_1 w$ and $P_2w$ denote perturbations to $q$-dynamics and $\theta$-dynamics, respectively, with exogenous variable $w$ being simply generated by an oscillator of the form
\[
\dot w = S w\,,\quad S=\begin{bmatrix}0 & 1 \cr -\Omega^2 & 0\end{bmatrix}\,.
\]
All other parameters are as in \cite{WangAut2019}. Suppose both $q$ and $\theta$ are measured periodically by sensors, with the sample period $T=0.1$. In this setting, the problem in question is to design an implementable regulator taking advantage of the sampled measurements such that all closed-loop signals are bounded and the regulation output $e(t):=\theta(t)$ asymptotically converges to zero.

Denote $y_m(t):=q(t)+\ell\theta(t)$, which is periodically available as $q(t)$ and $\theta(t)$ are measured periodically.
Thus, by setting $x:=\col(q+\ell\theta,\dot q+\ell\dot\theta,\theta,\dot\theta)$, we can rewrite (\ref{InPC}) in the form (\ref{ini-sys}) with
\[\ba{l}
A=\begin{bmatrix}0 &1 & 0 & 0 \cr 0 & 0 & g &0 \cr 0 &0&0& 1 \cr 0& \frac{\mu_f}{m_0\ell} & \frac{(m_0+m)g}{m_0\ell} & -\frac{\mu_f}{m_0} \end{bmatrix}\,,\quad B=\begin{bmatrix}0\cr 0 \cr 0 \cr {-1\over m_o\ell}\end{bmatrix} \,,\\
C_e = \begin{bmatrix}0 & 0 & 1 & 0\end{bmatrix}\,, \quad C_m = \begin{bmatrix}1 & 0 & 0 & 0\end{bmatrix}\,,\\  P^\top=\begin{bmatrix}0& {P_1^\top+P_2^\top\over m_0} & 0 & P_2^\top\over m_0\ell\end{bmatrix}\,,\quad Q_e=Q_m=0\,.
\ea\]
It is clear that the detectability property is fulfilled by the whole vector $(e,y_m)$, i.e., the pair $(A,C)$, instead of $(A,C_e)$ is detectable, with $C=\big[C_e^\top \,\, C_m^\top \big]^\top$.
Letting $m_0=0.5, m=2, \mu_f=0.2, g=9.8$, $\ell=0.3$ and $\Omega=5$, straightforward calculations show that Assumptions \ref{ass-1} and \ref{ass-2} are fulfilled.

Following the design paradigm proposed in Section \ref{sec-4}, we design the generalized zero-order hold device (\ref{zhd}) and the feedback law (\ref{u}) with
$L = \big[1 \,\, 0\big]$. Let $\Phi_D=\mbox{exp}(\Phi T)$ and $\Gamma_{\rm f} = \big[1 \,\, 0\big]$. We then design the compensator (\ref{IMPre}), and  the washout filter (\ref{filter}) via the $H_{\infty}$ control \cite{Green2012} with
\[
F_{\rm f}=\begin{bmatrix}0.5557&0.0959 \cr -1.55&0.8776\end{bmatrix}\,, G_{\rm f} = \begin{bmatrix}0.3219\cr -0.8471\end{bmatrix}\,.
\]
With the above design, the remaining problem is to design a discrete-time output feedback stabilizer for the corresponding discrete-time system (\ref{E-DTS}), (\ref{filter}), and (\ref{IMPre}),  which  can be easily solved via the $H_{\infty}$ control \cite{Green2012} again.

As seen from simulation results in Fig. \ref{fig2}, it can be seen that the regulation error $e(t)$ converges to zero and $y_m(t)$ is bounded.
Regarding a multi-rate digital regulator, the control action is executed with a period $\frac{T}{4}$. The simulation results are presented in Fig. \ref{fig3}, where both $e(t)$ and $y_m(t)$ are bounded. In contrast, we apply the emulation-based approach to design a digital regulator (\ref{eq:regulator-Em2}). When the sampling period is $T$, it is found that the trajectories of $e(t)$ and $y_m(t)$ are unbounded. When the sampling period is $\frac{T}{4}$, the simulation results are given in Fig. \ref{fig4}, where both $e(t)$ and $y_m(t)$ are bounded.

\begin{figure}[thpb]
\begin{center}
\centering\includegraphics[height=45mm,width=60mm]{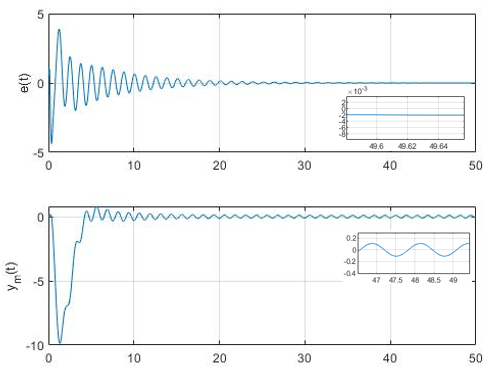} \caption{Trajectories of $e(t),y_m(t)$ with a generalized hold device.}
\label{fig2}
\end{center}
\end{figure}

\begin{figure}[thpb]
\begin{center}
\centering\includegraphics[height=45mm,width=60mm]{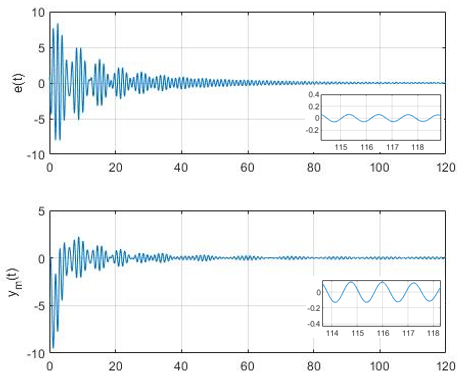} \caption{Trajectories of $e(t),y_m(t)$ with a multi-rate regulator.}
\label{fig3}
\end{center}
\end{figure}

\begin{figure}[thpb]
\begin{center}
\centering\includegraphics[height=45mm,width=60mm]{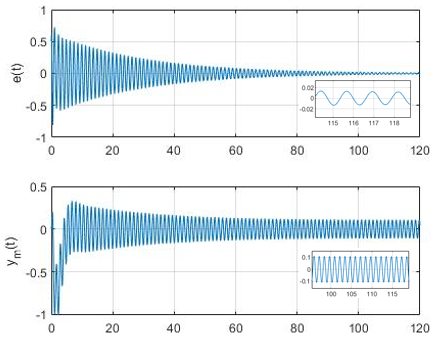} \caption{Trajectories of $e(t),y_m(t)$ using the emulation-based approach with sampling period $\frac{T}{4}$.}
\label{fig4}
\end{center}
\end{figure}

\section{Conclusion}
\label{sec-conclusion}

In this paper, the robust implementable output regulator design problem has been investigated for general linear continuous-time \mbox{systems} with periodically sampled measurements, consisting of both the regulation errors and extra measurements that are generally non-vanishing in steady state. We showed that the conventional emulation-based solution cannot be used to handle the problems when the sampling period is large and the asymptotic regulation is desired. Motivated by this, we proposed a design framework by incorporating a generalized zero-order hold device, which transforms the original problem into the problem of designing an output feedback controller fulfilling two conditions for a discrete-time system. With the design of a discrete-time compensator and a discrete-time washout filter, it has been shown that there always exists a discrete-time output feedback stabilizer for the resulting augmented system, which, together with  the previously designed compensator and  filter, completes the design of the controller. The resulting regulator structure aligns with the pre-processing scheme proposed in \cite{Wang2020}, by regarding the generalized zero-order hold device and the discrete-time compensator as an internal model.
Furthermore, the framework was generalized by proposing a multi-rate digital regulator, guaranteeing the regulation errors bounded by a constant, depending on the time derivative of the desired steady states of  control inputs, and adjustable by the control execution period.

\appendix

\section{Proof of Proposition \ref{proposition-1}}
\label{proof-pro-1}

The proof mainly follows the idea of \cite{Astolfi2018ECC}. Let
\[\ba{l}
\tilde x = x-\Pi_x w\,,\quad  \tilde x_c = x_c-\Pi_{x_c} w\,,\\
\tilde y = y-\hat y\,, \quad \tilde u = C_c { x}_c - u\,,
\ea\]
which rewrites the resulting closed-loop system into the form consisting of  (i) the flow dynamics
\[\ba{rcl}
\dot \tau &=& 1\,\\
\dot {\tilde x} &=& A\,{\tilde x} + B\, C_c {\tilde x}_c - B\, \tilde u \,\\
\dot {\tilde x}_c &=& A_c \tilde x_c + B_c C\tilde x - B_c\tilde y\,\\
\dot {\tilde y} &=& CA\tilde x + CBC_c\tilde x_c - C B\, \tilde u+ \begin{bmatrix}0\cr Y_mSw\end{bmatrix}\,\\
\dot{\tilde u} &=& C_c A_c \tilde x_c + C_c B_c C\tilde x - C_cB_c\tilde y + C_c\Pi_{x_c}Sw
\ea
\]
for $\tau\in[0,T)$, and (ii) the jump dynamics
\[
\ba{l}
\tau^+ = 0\,,\quad
\tilde x^+ = \tilde x\,,\quad
\tilde x_c^+ = \tilde x_c\,,\quad
\tilde y^+ = 0\,,\quad
\tilde u^+ = 0\,,\quad
\ea
\]
for $\tau=T$.
With $\tau_{\max}$ defined in (\ref{eq:tau_max}), we let $\lambda\in(0,1)$ and $\phi:[0,\tau_{\max}]\rightarrow\mathbb{R}$ be the solution of
  \[
  \dot\phi = -2\kappa\phi - \gamma(\phi^2+1)\,,\quad \phi(0) = \lambda^{-1}\,.
  \]
  According to \cite{Carnevale2007}, we have $\phi\in[\lambda,\lambda^{-1}]$. Then we can choose a Lyapunov function as
  \[
  V(\tau,\tilde x,\tilde x_c,\tilde y)=\begin{bmatrix}\tilde x \cr\tilde x_c\end{bmatrix}^\top\mathbf{P}\begin{bmatrix}\tilde x \cr\tilde x_c\end{bmatrix} + \phi(\tau) \begin{bmatrix}\tilde y \cr \tilde u\end{bmatrix}^\top\begin{bmatrix}\tilde y \cr \tilde u\end{bmatrix}
  \]
  which implies
  \[
  \dot V \leq -2\kappa V + \frac{2}{\gamma}\left\|\begin{bmatrix}Y_mSw\cr C_c\Pi_{x_c}Sw\end{bmatrix}\right\|^2
  \]
  during flow and $V^+ \leq V$  during jump.
  Therefore, it can be easily verified that the statement (i) is true, and
  \[
  \lim_{t+j\rightarrow\infty}\|\tilde x(t,j)\|^2 \leq \frac{1}{\gamma\kappa\|\mathbf{P}\|}\left\|\begin{bmatrix}Y_mSw\cr C_c\Pi_{x_c}Sw\end{bmatrix}\right\|^2\,,
  \]
  yielding (\ref{eq:emu-e}) with $k_{\rm em}={\|C_e\||\mathcal{W}|}/{\sqrt{\|{\bf P}\|}}$.
  This completes the proof. $\blacksquare$

\section{Proof of Theorem \ref{theo-1}}
\label{proof-theo-1}

By setting $\chi:=\col(x_D,\zeta_D,z)$, the resulting closed-loop (\ref{E-DTS}), (\ref{DR}) can be compactly described by
\beeq{\ba{rcl}
\dot w_D &=& 0\,,\quad \dot\chi =0\,,\\
w_{D}^+ &=& S_D \, w_D\,,\quad
\chi^+ = A_{cl} \chi + P_{cl} w_D
\ea}
for some appropriately defined matrices $A_{cl},P_{cl}$. With the requirement (a), it immediately follows that $|\sigma(A_{cl})|<1$, i.e., all eigenvalues of $A_{cl}$ lie within the unit circle. Thus there exists a {unique} $\Pi_{cl}\in\mathbb{R}^{(n+dq_e+n_z)\times d}$ such that
\beeq{\label{Pi-cl}
\Pi_{cl} S_D = A_{cl} \Pi_{cl} + P_{cl}\,.
}
With condition (b), there exists a solution $\Pi_z$ for the equations (\ref{RE-z}). With $\Pi_x,\Pi_\zeta$ being a solution of (\ref{REs-2}), it can be easily concluded that $\Pi_{cl}=\col(\Pi_x,\Pi_\zeta,\Pi_z)$ is a solution of (\ref{Pi-cl}), and thus the unique one.

Let $\rho:=\col(x,\zeta,e,y_m,z)$, which compactly expresses the hybrid system (\ref{ex-sys-f}), (\ref{ex-sys-j}), and (\ref{DR}) as the form
\beeq{\label{cl-hyb}\ba{rcl}
\dot w &=& Sw\,,\quad w^+ = w\,\\
\dot \rho &=& F_{cl} \rho + P_{F}w\,\\
\rho^+ &=& J_{cl} \rho + P_{J}w
\ea}
with some appropriately defined matrices $F_{cl},J_{cl},P_{F},P_{J}$, with  $|\sigma(J_{cl}e^{F_{cl}T})|<1$ by the requirement (a). Thus system (\ref{cl-hyb}) is exponentially stable at the invariant set $\mathcal{M}=\{(\tau,w,\rho): \rho=\widehat\Pi_{cl}(\tau) w\}$ with $\widehat\Pi_{cl}(\tau):[0,T)\rightarrow\mathbb{R}^{(n+dq_e+n_z)\times d}$ being the \emph{unique} solution of the equations
\beeq{\label{Re-cl-hy}\ba{rcl}
\frac{d \widehat\Pi_{cl}(\tau)}{d\tau} + \widehat\Pi_{cl}(\tau)S &=& F_{cl} \widehat\Pi_{cl}(\tau) + P_{F}\,\\
\widehat\Pi_{cl}(0) &=& J_{cl}\widehat\Pi_{cl}(T) + P_J\,.
\ea}
With (\ref{RE-3}), (\ref{RE-4}), and (\ref{RE-z}), simple calculations show that
\[
\widehat\Pi_{cl}(\tau):=\col(\Pi_x,\Pi_\zeta,0,Y_m\,e^{-S\tau},\Pi_z\,e^{-S\tau})
\]
is a solution of (\ref{Re-cl-hy}), and thus is the unique one. Since $C_e\Pi_x+Q_e=0$ in (\ref{RE-3}), it indicates that $e(t):=C_ex(t) + Q_ew(t)$ vanishes in  $\mathcal{M}$. The proof is thus completed. $\blacksquare$

\section{Proof of Lemma \ref{coro-1}}
\label{proof-coro-1}

The ``if" part has been proved in Theorem \ref{theo-1}. As for the proof of ``only if" part, we can see that the requirement (a) is clear. Thus, we now focus on the proof of the requirement (b). Using the notations in the proof of Theorem \ref{theo-1}, we use (\ref{cl-hyb}) to denote the resulting hybrid closed-loop system (\ref{ini-sys}), (\ref{zhd}), (\ref{u}), and (\ref{DR}). Simple calculations show that (\ref{Re-cl-hy}) has the unique solution $\hat\Pi_{cl}(\tau)$, which can be partitioned as
\[
\widehat\Pi_{cl}(\tau):=\col(\Pi_x(\tau),\Pi_\zeta(\tau),0,Y_m\,e^{-S\tau},\Pi_z\,e^{-S\tau})\,
\]
and
\beeq{\label{ro}
0 = C_e \Pi_x(\tau) + Q_e\,.
}
To be explicit, we can equivalently rewrite (\ref{Re-cl-hy}) as
\begin{equation}\label{18}
\ba{l}
\left\{
\ba{rcl}
\dot \Pi_x(\tau)  &=& - \Pi_x(\tau) S +A\,\Pi_x(\tau) + B\Psi(\tau) + P\,\\
\Pi_x(0) &=& \Pi_x(T)\,\\
\ea\right.\\
\dot \Pi_\zeta(\tau)  = - \Pi_\zeta(\tau) S + (\Phi\otimes I_{q_e})\Pi_\zeta(\tau) \,\\
\Pi_z S_D = A_z\, \Pi_z + B_z\, \begin{bmatrix}{\bf 0}_{d\times q_e} & Y_m^{\top} \end{bmatrix}^{\top}\,
\ea
\end{equation}
where  $Y_m = C_m \Pi_x(0) + Q_m$, and $\Psi(\tau)=\Psi_\zeta(\tau)+\Psi_{v_u}e^{-S\tau}$ with $\Psi_\zeta(\tau)=L\Pi_\zeta(\tau)$ and
\vspace{-2mm}
\[
\begin{bmatrix}\Psi_{v_u} \cr \Pi_\zeta(0)S_D\end{bmatrix} = K_z \,\Pi_z + L_z\, \begin{bmatrix}{\bf 0}_{d\times q_d} & Y_m^{\top}\end{bmatrix}^{\top}\,.
\vspace{-2mm}
\]

Putting (\ref{ro}) and the first part of (\ref{18}) together, by Assumption \ref{ass-1} and $m=q_e$, we observe that they reduce to the continuous-time regulator equations (\ref{RE-3}) and have the unique constant solution $\Pi_x(\tau),\Psi(\tau)$, i.e., independent of $\tau$. On the other hand, taking the second  of (\ref{18}) into consideration, we can deduce that $\Pi_\zeta$ is also independent of $\tau$ since it allows for a solution $\Pi_\zeta(\tau)$ if and only if $\dot\Pi_\zeta(\tau)=0$.  Furthermore,  we observe that $\Psi_{v_u}$ is constant, and satisfies $\Psi_{v_u} =[\Psi-\Psi_\zeta]e^{S\tau}$. To make this equality holds for constant $\Psi$ and $\Psi_\zeta$, there necessarily holds $\Psi=\Psi_\zeta$, leading to $\Psi_{v_u}=0$. In view of the previous observations,  the requirement (b) can be easily concluded by using the fact that $P,Q$ are arbitrary matrices. $\blacksquare$

\section{Proof of Lemma \ref{lemma-2}}
\label{app-lemma-2}
Instrumental to the subsequent analysis is the following lemma.
\begin{lemma}\label{lemma-1}
Suppose Assumptions \ref{ass-1} and \ref{ass-2} hold. Let $\mathcal{T}_1(\lambda)=\col(1,\lambda,\ldots,\lambda^{d-1})\otimes I_{q_e}$. Then
\beeq{\label{DNRC}
\mbox{rank}\begin{bmatrix}
A_D-\lambda I_n & L_D\mathcal{T}_1(\lambda) \cr C_e & 0
\end{bmatrix}
= n+q_e\,
}
holds for all $\lambda\in\sigma(S_D)$.

\end{lemma}
\begin{proof}
Consider the auxiliary system
\[\ba{l}
\dot \xi = (\Phi\otimes I_{q_m}) \xi \,,\quad
\dot x = A x + BL \zeta \,,\quad
\dot \zeta = (\Phi\otimes I_{q_e}) \zeta \,\\
e = C_e x\,,\quad
y_m = C_m x - \Gamma_{\rm f}\xi
\ea
\]
which, by Assumption 1 and with the construction that $(\Phi\otimes I_{q_m},\Gamma_{\rm f})$ and $(\Phi\otimes I_{q_e},L)$ are observable, respectively, can be easily inferred to be detectable. Thus, with Assumption \ref{ass-2} and according to \cite{Kimura1990} again, it can be seen that its discretized form, as
\[\ba{l}
\xi_{k+1} = (\Phi_D\otimes I_{q_m}) \xi_k\,,\quad
x_{k+1} = A_D x + L_D \zeta_k \,\quad\,\\
\zeta_{k+1} = (\Phi_D\otimes I_{q_e}) \zeta_k \,\\
e_k = C_e x_k\,,\quad
y_{m,k}= C_m x_k - \Gamma_{\rm f}\xi_k
\ea
\]
must also be detectable, which, based on the PBH test, is equivalent to saying that the matrix
\[\left[
\ba{ccc}
(\Phi_D-\lambda I_d)\otimes I_{q_m} & 0 & 0\cr
0&A_D -\lambda I_n & L_D\, \cr
0&0& (\Phi_D-\lambda I_d)\otimes I_{q_e} \cr \hline
0& C_e & 0\,\cr
- \Gamma_{\rm f} & C_m &0
\ea
\right]
\]
is full-column-rank for all $\lambda\in\{\lambda\in\mathbb{C} |\, |\lambda| \geq 1\}$. Then considering $\lambda \in \sigma(\Phi_D)$, (\ref{DNRC}) can be concluded by some simple column transformation. $\blacksquare$
\end{proof}

We now proceed to use the PBH test and Lemma \ref{lemma-1} to verify the stabilizability and detectability of (\ref{E-DTS}).
Regarding the stabilizability, it holds if and only if all rows of the following matrix 
\[
\left[
\underbrace{
\ba{cccc}
F_{\rm f}-\lambda I & G_{\rm f}\,C_m & 0 & 0\, \cr
0& A_D-\lambda I & L_D & 0  \cr
0& 0 & -\lambda I  & I \cr
0& 0 & 0& (\Phi_D-\lambda I)\otimes I_{q_e}
\ea}_{\hat A}
\left|
\underbrace{\ba{ccc}
0&0&0\, \cr
B_D&0&0 \cr
0 &I&0\cr
0&0&I
\ea}_{\hat B}\right.
\right]
\]
are independent for all $\lambda\in\{\lambda\in\mathbb{C} |\, |\lambda| \geq 1\}$. We note that  $F_{\rm f}-\lambda I$ is nonsingular for all $\lambda\in\{\lambda\in\mathbb{C} |\, |\lambda| \geq 1\}$ by construction and  $(A_D,B_D)$ is stabilizable by \cite{Kimura1990} and Assumptions \ref{ass-1}.(i). Thus, it can be easily verified that the above matrix is full-row-rank
for all $\lambda\in\{\lambda\in\mathbb{C} |\, |\lambda| \geq 1\}$.

To further explore the detectability,  it is true if and only if for all $\lambda\in\{\lambda\in\mathbb{C} |\, |\lambda| \geq 1\}$, the matrix
\[
\left[\ba{l}
  \hat A \\ \hline
  \hat C\ea
\right]\,, \quad \mbox{with }
\hat C = \begin{bmatrix}
0 & C_e & 0 & 0\,\cr
-\Gamma_{\rm f} & C_m &0 &0
\end{bmatrix}
\]
is full-column-rank. Since $F_{\rm f}+G_{\rm f}\Gamma_{\rm f}=\Phi_D\otimes I_{q_m}$ by construction, the above verification reduces to show
\[
\left[\ba{cccc}
(\Phi_D-\lambda I)\otimes I_{q_m} & 0 & 0 & 0\, \cr
0 & A_D -\lambda I & L_D & 0\, \cr
0 & 0& -\lambda I & I\cr
0 & 0& 0 & (\Phi_D-\lambda I)\otimes I_{q_e}\cr \hline
0 & C_e & 0 & 0\,\cr
-\Gamma_{\rm f} & C_m &0 &0
\ea
\right]\]
is full-column-rank for all $\lambda\in\{\lambda\in\mathbb{C} |\, |\lambda| \geq 1\}$. For all $\lambda\in\{\lambda\in\mathbb{C} |\, |\lambda| \geq 1, \lambda\notin\sigma(\Phi_D)\}$, the above matrix is full-column-rank if and only if $(A_D,C)$ is detectable, which is clearly true by \cite{Kimura1990} and Assumptions \ref{ass-1}.(i).

With this being the case, we turn to investigate the case that $\lambda\in\sigma(\Phi_D)$. Since $(\Phi_D\otimes I_{q_m},\Gamma_{\rm f})$ is observable by construction, by taking appropriate column transformation, the previous verification reduces to show
\[
\mbox{rank}\begin{bmatrix}
 A_D -\lambda I & L_D\mathcal{T}_1(\lambda)\, \cr
 C_e & 0\,\cr
\end{bmatrix}
= n + q_e \, \quad \forall \lambda\in\sigma(\Phi_D)\,,
\]
which clearly is true by recalling (\ref{DNRC}) and the fact that $\sigma(\Phi_D)=\sigma(S_D)$.
The proof is thus completed. $\blacksquare$

\section{Proof of Theorem \ref{theo-3}}
\label{app-theo-3}

Let $\tilde x_D:=x_D-\Pi_xw_D$, $\tilde\zeta_D:=\zeta_D-\Pi_\zeta w_D$, and $\tilde z := z-\Pi_z w_D$.
The closed-loop discrete-time system (\ref{E-DTS})-(\ref{DR}) can be compactly rewritten as
\[\ba{l}
\dot{\tau} = 0\,,\qquad \dot{\tilde X}_{cl} = 0\,,\quad \qquad \,\,\mbox{for $\tau\in[0,T)$}\\
{\tau}^+ = 0\,,\,\quad {\tilde X}_{cl}^+ = A_{cl}{\tilde X}_{cl}\,,\quad \mbox{for $\tau = T$}
\ea\]
with $\tilde X_{cl}:=\col(\tilde x_D,\tilde \zeta_D,\tilde z)$ and
\[
A_{cl}: = \begin{bmatrix}
           A_{D}+B_DL_{z,u} & L_{D} & K_{z,u} \cr
           L_{z,\zeta}C & 0 & K_{z,\zeta} \cr
           B_zC & 0 & A_z \cr
    \end{bmatrix}\,.
\]
With the controller (\ref{DR}) satisfying Theorem \ref{theo-1}, there exists a symmetric positive definite matrix $P_{cl}$ and a positive constant $\lambda_{cl}<1$ such that $V_{cl} = \tilde X_{cl}^\top P_{cl}\tilde X_{cl}$  satisfies
\beeq{\label{eq:V_cl}
\dot V_{cl} = 0\,,\qquad
V_{cl}^+ = \tilde X_{cl}^\top A_{cl}^\top P_{cl}A_{cl}\tilde X_{cl}\leq \lambda_{cl} V_{cl}\,.
\vspace{-3mm}
}
To ease the subsequent analysis, we propose a clock $\tau_{cl}$ dynamics as
\[\ba{ll}
\dot \tau_{cl} = 1\,,\quad &\mbox{for }\,  \tau_{cl}\in\bigcup_{i=1}^N[0,\frac{i}{N}T)\\
\tau_{cl}^+ = (1 - \lfloor\frac{\tau_{cl}}{T}\rfloor)\tau\,,\quad &\mbox{for } \lfloor\frac{N\tau_{cl}}{T}\rfloor\in\mathbb{N}_+\,.
\ea\]
With such a clock, we set $\varepsilon := u-\hat u$, and can rewrite the closed-loop system (\ref{ini-sys})-(\ref{eq:control})-(\ref{eq:hat_u}) under the error coordinates $\tilde X_{cl}=\col(\tilde x_D,\tilde \zeta_D,\tilde z)$ as
\[
\left\{
\ba{l}
\dot \tau_{cl} = 1\,\\
\dot w_D = 0\,\\
\dot{{\tilde X}}_{cl} = \begin{bmatrix} e^{-A\tau_{cl}}B\varepsilon\cr {\bf 0}_{dq_e+n_z} \end{bmatrix}\,\\
\dot\varepsilon = -L(\Phi\otimes I_{q_e})e^{(\Phi\otimes I_{q_e})\tau_{cl}}\tilde\zeta_D - L(\Phi\otimes I_{q_e})\Pi_\zeta e^{\tau_{cl} S}w_D\,\\
\ea\right.
\]
for $\tau_{cl}\in\bigcup_{i=1}^N[0,\frac{i}{N}T)$, and
\[
\left\{
\ba{rcl}
\tau_{cl}^+ &=& (1 - \lfloor\frac{\tau_{cl}}{T}\rfloor)\tau_{cl}\,\\
w_D^+ &=& [(1-\lfloor\frac{\tau_{cl}}{T}\rfloor) + \lfloor\frac{\tau_{cl}}{T}\rfloor S_D ]\, w_D\,\\
{\tilde X}_{cl}^+
&=&\bigg((1-\lfloor\frac{\tau_{cl}}{T}\rfloor) I + \lfloor\frac{\tau_{cl}}{T}\rfloor A_{cl}\bigg){\tilde X}_{cl}\,\\
\varepsilon^+ &=& 0
\ea\right.
\]
for $\lfloor\frac{N\tau_{cl}}{T}\rfloor\in\mathbb{N}_+$.

Let $k_1=\frac{1}{T}\ln \frac{1}{\lambda_{cl}}$, and
\[
\alpha^\ast = \mbox{argmin}_{\alpha\in\mathbb{R}_+}\frac{32\phi_1^2e^\alpha + 8\phi_2\lambda_{cl}}{\alpha k_1^2\sigma_m(P_{cl})\lambda_{cl}^2}\,
\]
with $\phi_1=\max\limits_{{\tau_{cl}}\in[0,T]}\|e^{-A{\tau_{cl}}}B\|\|P_{cl}\|$ and $\phi_2=\max\limits_{{\tau_{cl}}\in[0,T]}\|L\|\|\Phi e^{\Phi{\tau_{cl}}}\|$.
Then set $k_2:=\alpha^\ast k_1 N$ and
\[
N \geq N^\ast:=\frac{32\phi_1^2e^{\alpha^\ast} + 8\phi_2\lambda_{cl}}{\alpha^\ast k_1^2\sigma_m(P_{cl})\lambda_{cl}^2}\,,
\]
and choose a Lyapunov function as $U = e^{-k_1{\tau_{cl}}} V_{cl} +  e^{-k_2 s} \|\varepsilon\|^2$ with $s={\tau_{cl}}-\lfloor\frac{N{\tau_{cl}}}{T}\rfloor\frac{T}{N}$. It is clear that $s\leq \frac{T}{N}$, and $\dot s=1$ for ${\tau_{cl}}\in\bigcup_{i=1}^N[0,\frac{i}{N}T)$ and $s^+=0$ for $\lfloor\frac{N{\tau_{cl}}}{T}\rfloor\in\mathbb{N}_+$.

During jumps, we consider two cases: (i) ${\tau_{cl}} <T$ and (ii) ${\tau_{cl}}=T$. For ${\tau_{cl}} <T$, we have
$U^+ = e^{-k_1{\tau_{cl}}} V_{cl} \leq U$.
For ${\tau_{cl}}=T$, by (\ref{eq:V_cl}), we have
$U^+ = \lambda_{cl}\widetilde X_{D}^\top P_{cl}\widetilde X_{D} \leq U$.
Hence, during jumps we always have $U^+ \leq U$.

During flows, with (\ref{RE-4}) and (\ref{eq:V_cl}) we compute the time derivative of $U$ as
\[\ba{l}
\dot U \leq -k_1e^{-k_1{\tau_{cl}}} \tilde X_{cl}^\top P_{cl}\tilde X_{cl} + 2 e^{-k_1{\tau_{cl}}} \|\tilde X_{cl}^\top P_{cl}\|\|e^{-A{\tau_{cl}}}B\varepsilon\|\,  \,\\ \quad -k_2e^{-k_2 s} \|\varepsilon\|^2
  + 2e^{-k_2 s} \|\varepsilon\|(\phi_2\|\tilde\zeta_D\|+ \|L(\Phi\otimes I_{q_e})\Pi_\zeta w\|)\,\\
    \leq -(k_1e^{-k_1{\tau_{cl}}} \sigma_{m}(P_{cl})-\frac{4\phi_2}{k_2})\|\tilde X_{cl}\|^2  + 2 e^{-k_1{\tau_{cl}}} \phi_1\|\tilde X_{cl}\|\|\varepsilon\|\,  \,\\ \quad -\frac{k_2}{2}e^{-k_2 \frac{T}{N}} \|\varepsilon\|^2
  + \frac{4}{k_2}\| \Psi S w\|^2\,\\
  \leq -\frac{k_1}{2}e^{-k_1T} \sigma_{m}(P_{cl})\|\tilde X_{cl}\|^2 -\frac{k_2}{4}e^{-k_2 \frac{T}{N}} \|\varepsilon\|^2  + \frac{4}{k_2}\| \Psi S w\|^2\,.
\ea\]
Thus, we have
\beeq{\label{eq:Xcl}\ba{l}
\lim\limits_{t+j\rightarrow\infty}\|\varepsilon(t,j)\|^2\leq \frac{4 e^{\alpha^\ast}}{{\alpha^\ast}^2 k_1^2N^2\lambda_{cl}}\| \Psi S w\|^2\,\\
\lim\limits_{t+j\rightarrow\infty}\|\tilde X_{cl}(t,j)\|^2 \leq  \frac{8}{\alpha^\ast k_1^2N \lambda_{cl}\sigma_m(P_{cl})}\| \Psi S w\|^2\,.
\ea}
This then yields that the trajectories of the resulting closed-loop system (\ref{ini-sys})-(\ref{eq:control})-(\ref{eq:hat_u}) are bounded.

Regarding the bound of $\|e\|$, we note that by (\ref{eq:def-xD}) and (\ref{REs-2}), we have $e=C_e\tilde x$ with
\vspace{-3mm}
\[\ba{l}
\tilde x  =  e^{A{\tau_{cl}}}\tilde x_D+ \int_0^{\tau_{cl}} e^{-A\,(r-{\tau_{cl}})}BL e^{(\Phi\otimes I_{q_e})r}{\rm d} r \,\tilde\zeta_D \\ \qquad + \int_0^{\tau_{cl}} e^{-A\,(r-{\tau_{cl}})}{\rm d} r \,B \,(K_{z,u}\tilde z + L_{z,u}\tilde x_D)  \,.
\ea\]
Thus, using the latter of (\ref{eq:Xcl}) and ${\tau_{cl}}\in[0,T]$, it can be verified that there exists a $\gamma^\dag$ such that (\ref{eq:e-bound}) holds. $\blacksquare$


\end{document}